\documentclass[10pt]{article}
\usepackage{color,latexsym,amsmath,amssymb,path,url,amsthm,enumerate,verbatim,hyperref}
\usepackage[margin=1in]{geometry}

\newcommand{\ignore}[1]{}

\makeatletter
\newtheorem*{rep@theorem}{\rep@title}
\newcommand{\newreptheorem}[2]{%
\newenvironment{rep#1}[1]{%
 \def\rep@title{#2 \ref{##1}}%
 \begin{rep@theorem}}%
 {\end{rep@theorem}}}
\makeatother
\newreptheorem{theorem}{Theorem}

\newtheorem{theorem}{Theorem}[section]
\newtheorem{example}{Example}
\newtheorem{lemma}[theorem]{Lemma}
\newtheorem{corollary}[theorem]{Corollary}

\newtheorem{conjecture}{Conjecture}

\newtheorem{definition}[theorem]{Definition}
\newcommand{\NN}{\ensuremath{\mathbb N}}

\newcommand{\RR}{\ensuremath{\mathbb R}}
\newcommand{\CC}{\ensuremath{\mathbb C}}
\newcommand{\PP}{\ensuremath{\mathbb P}}
\newcommand{\FF}{\ensuremath{\mathbb F}}

\def\arg#1{\mathtt{arg}\left(#1\right)}

\def\supp{\mathtt{supp}}
\def\rank{\mathtt{rank}}
\def\adim{\mathtt{affine}\mbox{-}\mathtt{dim}}
\def\dim{\mathtt{dim}}

\renewcommand{\L}{\mathcal L}
\newcommand{\V}{\mathcal V}

\title{On the number of ordinary lines determined by sets in complex space}
\author{Abdul Basit\thanks{Department of Computer Science, Rutgers University.
Email: \texttt{abasit@cs.rutgers.edu}.}\and
Zeev Dvir\thanks{Princeton University, Department of Mathematics and Department of Computer Science. Research supported by NSF CAREER award DMS-1451191 and NSF grant CCF-1523816.
Email: \texttt{zeev.dvir@gmail.com}.}\and
Shubhangi Saraf\thanks{Department of Computer Science and Department of Mathematics, Rutgers University. Research supported in part by NSF grants CCF-1350572 and CCF-1540634.
Email: \texttt{shubhangi.saraf@gmail.com}.}\and
Charles Wolf\thanks{Department of Mathematics, Rutgers University. Research supported in part by NSF grant CCF-1350572.
Email: \texttt{charlesi@post.bgu.ac.il}.}
}

\begin{document}
\pagenumbering{arabic}

\maketitle

\begin{abstract}
	Kelly's theorem states that a set of $n$ points affinely spanning $\mathbb{C}^3$ must determine at  least one ordinary complex line (a line incident to exactly two of the points). Our main theorem shows that such sets determine at least $3n/2$ ordinary lines, unless the configuration has $n-1$ points in a plane and one point outside the plane (in which case there are at least $n-1$ ordinary lines). In addition, when at most $n/2$ points are contained in any plane, we prove  stronger bounds that take advantage of the existence of lines with four or more points (in the spirit of Melchior's and Hirzebruch's inequalities). Furthermore, when the points span four or more dimensions, with at most $n/2$ points contained in any three dimensional affine subspace, we show that there must be a quadratic number of ordinary lines.
\end{abstract}

%%%%%

\section{Introduction}

Let $\V = \{v_1, v_2, \dots, v_n\}$ be a set of $n$ points in $\CC^d$. We denote by $\L(\V)$ the set of lines determined by points in $\V$,  and by $\L_r(\V)$ (resp. $\L_{\geq r}(\V)$) the set of lines in $\L(\V)$ that contain exactly (resp. at least) $r$ points. Let $t_r(\V)$ denote the size of $\L_r(\V)$. Throughout the write-up we omit the argument $\V$ when the context makes it clear. We refer to $\L_2$ as the set of {\em ordinary lines}, and $\L_{\geq 3}$ as the set of {\em special lines}.

A well known result in combinatorial geometry is the Sylvester-Gallai theorem. 
\begin{theorem}[Sylvester-Gallai theorem]\label{th:SG}
Let $\V$ be a set of $n$ points in $\RR^2$ not all on a line.  Then there exists an ordinary line determined by points of $\V$.
\end{theorem}
The statement was conjectured by Sylvester in 1893~\cite{syl1893}, and the first published proof is by Melchior~\cite{mel40}. Later proofs were given by Gallai in 1944~\cite{gal44} and others; there are now several different proofs of the theorem. Of particular interest is the following result by Melchior~\cite{mel40}.
\begin{theorem}[Melchior's inequality]\label{th:Melchior}
Let $\V$ be a set of $n$ points in $\RR^2$ that are not collinear.  Then
\[ t_2(\V) \geq 3 + \sum \limits_{r \geq 4} (r - 3) t_r(\V). \]
\end{theorem}
Theorem~\ref{th:Melchior} in fact proves something stronger than the Sylvester-Gallai theorem, i.e., there are at least three ordinary lines. A natural question to ask is how many ordinary lines must a set of $n$ points, not all on a line, determine. This led to what is known as the {\em Dirac-Motzkin conjecture}.
\begin{conjecture}[Dirac-Motzkin conjecture]
For every $n \neq 7, 13$, the number of ordinary lines determined by $n$ non-collinear points in the plane is at least $\left\lceil \frac{n}{2} \right\rceil$.
\end{conjecture}
There were several results on this question (see~\cite{motzkin51, kelly58, cs93}), before Green and Tao~\cite{gt13} resolved it for large enough point sets.
\begin{theorem}[Green and Tao~\cite{gt13}]\label{Green-Tao}
Let $\V$ be a set of $n$ points in $\RR^2$, not all on a line.  Suppose that $n \geq n_0$ for a sufficiently large absolute constant $n_0$. Then $t_2(\V) \geq \frac{n}{2}$ for even $n$ and $t_2(\V) \geq \left\lfloor \frac{3n}{4}\right\rfloor$ for odd $n$.
\end{theorem}
There are several survey articles on the topic, see for example \cite{bm90}, and \cite{gt13} provides a nice history of the problem.

The Sylvester-Gallai theorem is not true when the field $\RR$ is replaced by $\CC$. The well known Hesse configuration, realized by the nine inflection points of a non-degenerate cubic, provides a counterexample. A more general example is the following:
\begin{example}[Fermat configuration]\label{eg:fermat-example}
For any positive integer $k \geq 3$, let $\V$ be inflection points of the Fermat Curve $X^k + Y^k + Z^k = 0$ in $\PP\CC^2$. Specifically, let
$$ \V_1 = \bigcup\limits_{i=1}^{k}\{[0:1:\omega^i]\}, \quad \V_2 = \bigcup\limits_{i=1}^{k}\{[\omega^i:0:1]\} , \quad \mbox{ and } \quad \V_3 = \bigcup\limits_{i=1}^{k}  \{[1:\omega^i:0]\}, $$
where $\omega$ is a $k^{th}$ root of $-1$. Then $\V = \V_1 \cup \V_2 \cup \V_3$.

The point set $\V$ does not determine any ordinary lines. It is easy to check that the sets $\V_1$, $\V_2$ and $\V_3$ lie on the lines $X = 0$, $Y = 0$, and $Z~=~0$, respectively. Every other line contains exactly three points. Indeed, consider, without loss of generality, points $u = [0:1:\omega^i] = [0:\omega^{-i}:1] \in \V_1$ and $v = [\omega^j:0:1] \in \V_2$. Then the point $w = [-\omega^{j}: \omega^{-i} : 0] = [1: -\omega^{-i-j}:0] = [1: \omega^{k-i-j}:0] \in \V_3$ is collinear with $u$ and $v$. 
\footnote{We note that the while Fermat configuration as stated lives in the projective plane, it can be made affine by any projective transformation that moves a line with no points to the line at infinity.}
\end{example}
In response to a question of Serre~\cite{ser66}, Kelly~\cite{kelly86} showed that when the points span more than two dimensions, the point set must determine at least one ordinary line.
\begin{theorem}[Kelly's theorem~\cite{kelly86}]\label{th:Kelly}
Let $\V$ be a set of $n$ points in $\CC^3$ that are not contained in a plane.  Then there exists an ordinary line determined by points of $\V$.
\end{theorem}
Kelly's proof of Theorem~\ref{th:Kelly} used a deep result of Hirzebruch~\cite{hir83}, referred to as Hirzebruch's inequality, from algebraic geometry. More elementary proofs of Theorem~\ref{th:Kelly} were given in \cite{eps06} and \cite{dsw14}. To the best of our knowledge, no lower bound greater than one is known for the number of ordinary lines determined by point sets spanning $\CC^3$. Improving on the techniques of \cite{dsw14}, we make the first progress in this direction.
\begin{theorem}\label{th:3n/2}
Let $\V$ be a set of $n \geq 24$ points in $\CC^3$ not contained in a plane. Then $\V$ determines at least $\frac{3}{2}n$ ordinary lines, unless $n - 1$ points are on a plane in which case there are at least $n - 1$ ordinary lines.
\end{theorem}
In the latter case, let $v$ be the point not on the plane. Every line determined by $\V$ containing $v$ is ordinary, so there are at least $n-1$ ordinary lines. On the other hand, it is possible that there are exactly $n - 1$ ordinary lines. In particular, let $\V$ consist of the Fermat Configuration, for some $k \geq 3$, on a plane, and one point $v$ not on the plane. Then $\V$ has $3k + 1$ points, and the only ordinary lines determined by $\V$ are lines that contain $v$, so there are exactly $3k$ ordinary lines. We are not aware of any examples that achieve the $\frac{3}{2}n$ bound when at most $n - 2$ points are contained in any plane. 
Using a similar argument, for point sets in $\RR^3$, Theorems~\ref{Green-Tao} and \ref{th:3n/2} give us the following easy corollary.
\begin{corollary}
Let $\V$ be a set of $n$ points in $\RR^3$ not contained in a plane.  Suppose that $n \geq n_0$ for a sufficiently large absolute constant $n_0$. Then $\V$ determines at least $\frac{3}{2}(n - 1)$ ordinary lines.
\end{corollary}

Hirzebruch's inequality is similar in spirit to Melchior's inequality and gives bounds on the number of ordinary lines by taking into consideration the number of lines with four or more points.
\begin{theorem}[Hirzebruch's inequality~\cite{hir83}]\label{th:Hirzebruch}
Let $\V$ be a set of $n$ points in $\CC^2$, such that $t_n(\V) = t_{n-1}(\V) = t_{n-2}(\V) = 0$.  Then
$$ t_2(\V) + \frac{3}{4}t_3(\V) \geq n + \sum_{r \geq 5} (2r - 9) t_r(\V). $$
\end{theorem}
More recently, Bojanowski~\cite{bojanowski03} and Pokora~\cite{pokora16} used a theorem of Langer~\cite{langer03} to prove the following result. 
\begin{theorem}[Langer's Inequality]\label{th:langer}
Let $\V$ be a set of $n$ points in $\CC^2$, such that $t_j(\V) = 0$ for $j > 2n/3$.  Then
$$ t_2(\V) + \frac{3}{4}t_3(\V) \geq n + \sum_{r \geq 5} \frac{r^2 - 4r}{4} t_r(\V). $$
\end{theorem}
While imposing a stricter condition on the point set, Theorem~\ref{th:langer} is strictly better than Theorem~\ref{th:Hirzebruch} as long as there is a line containing at least 5 points. For more detail on Langer's inequality and combinatorial implications, see~\cite{dezeeuw18b}. In $\CC^3$, when $\V$ is sufficiently non-degenerate, i.e., no plane contains too many points, we are able to give a more refined bound by taking into account the existence of lines with more than three points. In particular, we show the following (the constant $1/2$ in Theorem~\ref{th:main} is arbitrary and can be replaced by any positive constant smaller than 1):
\begin{theorem}\label{th:main}
There exists an absolute constant $c > 0$ and a positive integer $n_0$ such that the following holds. Let  $\V$ be a set of $n \geq n_0$ points in $\CC^3$ with at most $\frac{1}{2} n$ points contained in any plane. Then
\[ t_2(\V) \geq \frac{3}{2}n + c \sum_{r \geq 4} r^2 t_r(\V). \]
\end{theorem}

Suppose that $\V$ consists of $n - k$ points on a plane, and $k$ points not on the plane. There are at least $n - k$ lines incident to each point not on the plane, at most $k - 1$ of which could contain three or more points. So we get that there are at least $k(n - 2k)$ ordinary lines determined by $\V$. Then if $k = \epsilon n$, for $0 < \epsilon < 1/2$, we get that $\V$ has $\Omega_{\epsilon}(n^2)$ ordinary lines, where the hidden constant depends on $\epsilon$. Therefore, the bound in Theorem~\ref{th:main} is only interesting when no plane contains too many points.

On the other hand, having at most a constant fraction of the points on any plane is necessary to obtain a bound as in Theorem~\ref{th:main}. Indeed, let $\V$ consist of the Fermat Configuration for some $k \geq 3$ on a plane and $o(k)$ points not on the plane. Then $\V$ has $O(k)$ points and determines $o(k^2)$ ordinary lines. On the other hand, $\sum_{r \geq 4} r^2 t_r(\V) = \Omega(k^2)$.

Both Hirzebruch's and Langer's inequalities (which also give bounds in $\CC^3$, though without requiring that every plane contains not too many points) only give lower bounds on $t_2(\V) + \frac{3}{4}t_3(\V)$, whereas both Theorems~\ref{th:3n/2} and \ref{th:main} give lower bounds on the number of ordinary lines, i.e., $t_2(\V)$. We also note that lines with four points do not play any role in the previous inequalities, where the summation starts at $r = 5$. This is not the case for Theorem~\ref{th:main}, which also takes into account lines with four points.

Finally, when a point set $\V$ spans four or more dimensions in a sufficiently non-degenerate manner, i.e., no three dimensional affine subspace contains too many points, we prove that there must be a quadratic number of ordinary lines.
\begin{theorem}\label{th:higherdim}
There exists an absolute constant $c^\prime > 0$ and a positive integer $n_0$ such that the following holds. Let  $\V$ be a set of $n \geq n_0$ points in $\CC^4$ with at most $\frac{1}{2} n$ points contained in any three dimensional affine subspace. Then
\[ t_2(\V) \geq c^\prime n^2. \]
\end{theorem}
Here, again, the constant $1/2$ is arbitrary and can be replaced by any positive constant less than 1.  However, increasing this constant will shrink the constant $c^\prime$ in front of $n^2$.

While we state Theorems~\ref{th:3n/2} and \ref{th:main} over $\CC^3$ and Theorem \ref{th:higherdim} over $\CC^4$, the same bounds hold in higher dimensions as well, since projecting a point set in $\CC^d$ onto a generic lower dimensional subspace can only increase the number of incidences. A quadratic lower bound may also be possible under the weaker hypothesis of at most $\frac{1}{2}n$ points are contained in any two dimensional space, but we have no proof or counterexample.  A recent result of de Zeeuw~\cite{dezeeuw2018a} resolves this question for points over the reals.
\begin{theorem}
For every $\alpha < 1$ there is a $c_\alpha > 0$ such that, if a set $P$ of $n$ points in $\RR^3$
has at most $\alpha n$ points on any plane, then $P$ spans at least $c_\alpha n^2$ ordinary lines.
\end{theorem}

\paragraph{Organization:} In Section~\ref{sec:overview} we give a short overview of the new ideas in our proof (which builds upon \cite{dsw14}). In Section~\ref{sec:prelims} we develop the necessary machinery on matrix scaling and Latin squares. In Section~\ref{sec:dependencymatrixprelim}, we prove some key lemmas that will be used in the proofs of our main results. Section~\ref{sec:3n/2} gives the proof of Theorems~\ref{th:3n/2} and \ref{th:higherdim}, which are considerably simpler than Theorem~\ref{th:main}.  In Section~\ref{sec:dependencymatrix}, we develop additional machinery needed for the proof of Theorem~\ref{th:main}. The proof of Theorem~\ref{th:main} is presented in Section~\ref{sec:main}.

\section{Proof overview}\label{sec:overview}

The starting point for the proofs of Theorems~\ref{th:3n/2}, \ref{th:main} and \ref{th:higherdim} is the method developed in \cite{bdwy13,dsw14} which uses rank bounds for {\em design matrices} -- matrices in which the supports of different columns do not intersect in too many positions. We augment the techniques in these papers in several ways which give us more flexibility in analyzing the number of ordinary lines. We devote this short section to an overview of the general framework (starting with \cite{dsw14}) outlining the places where new ideas come into play.

Let $\V = \{v_1, \dots, v_n\}$ be  points in $\CC^d$ and denote by $V$  the $n \times (d+1)$ matrix whose $i^{th}$ row is the vector $(v_i, 1) \in \CC^{d+1}$, i.e., the vector obtained by appending a 1 to the vector $v_i$. The dimension of the (affine) space spanned by the point set can be seen to be equal to $\rank(V) - 1$. We would now like to argue that too many collinearities in $\V$ (or too few ordinary lines) imply that all (or almost all) points of $\V$ must be contained in a low dimensional affine subspace, i.e., $\rank(V)$ is small. To do this, we construct a matrix $A$, encoding the dependencies in $\V$, such that $AV = 0$. Then we must have \[ \rank (V) \leq n - \rank(A), \] and so it suffices to lower bound the rank of $A$.

We construct the matrix $A$ in the following manner so that each row of $A$ corresponds to a collinear triple in $\V$.  For any collinear triple  $\{ v_i, v_j, v_k \}$, there exist coefficients $a_i, a_j, a_k$ such that $a_i v_i + a_j v_j + a_k v_k = 0$. We can thus form a row of $A$ by taking these coefficients as the nonzero entries in the appropriate columns. By carefully selecting the triples using constructions of Latin squares (see Lemma~\ref{le:triplesystem}), we can ensure that $A$ is a {\em design matrix}. Roughly speaking, this means that the supports of every two columns in $A$ intersect in a small number of positions. Equivalently, every pair of points appears together only in a small number of triples. 

The proof in \cite{dsw14} now proceeds to prove a general rank lower bound on any such design matrix. To understand the new ideas in our proof, we need to ``open the box'' and see how the rank bound from \cite{dsw14} is actually proved. To get some intuition, suppose that $A$ is a matrix with 0/1 entries. To bound the rank of $A$, we can consider the matrix $M = A^{*}A$ and note that $\rank(M) = \rank(A)$. Since $A$ is a design matrix, $M$ has the property that the diagonal entries are very large (since we can show that each point is in many collinear triples) and  that the off-diagonal elements are very small (since columns have small intersections). Matrices with this property are called {\em diagonal dominant} matrices, and it is easy to lower bound their rank using trace inequalities (see Lemma~\ref{le:rankbound}). 

However, the matrix $A$ that we construct could have  entries of arbitrary magnitude  and so bounding the rank requires more work. To do this, \cite{dsw14} relies on {\em matrix scaling} techniques. We are allowed to multiply each row and each column of $A$ by a nonzero scalar and would like to reduce to the case where the entries of $A$ are ``mostly balanced'' (see Theorem~\ref{th:realscaling} and Corollary~\ref{co:complexscaling}). Once scaled, we can consider $M = A^*A$ as before and use the bound for diagonal dominant matrices.

Our proof introduces two new main ideas into this picture. The first idea has to do with the conditions needed to scale $A$. It is known (see Corollary~\ref{co:complexscaling}) that a matrix  $A$ has a good scaling if it does not contain a ``too large'' zero submatrix. This is referred to as having  Property-$S$ (see Definition~\ref{def:propertys}). The proof of \cite{dsw14} uses $A$ to construct a new matrix $B$, whose rows are the same as those of  $A$ but with some rows repeating more than once. Then one shows that $B$ has Property-$S$ and continues to scale $B$ (which has rank at most that of $A$) instead of $A$. This loses the control on the exact number of rows in $A$ which is crucial for bounding the number of ordinary lines. We instead perform a more careful case analysis: If $A$ has Property-$S$ then we scale $A$ directly and gain more information about the number of ordinary lines. If $A$ does not have Property-$S$, then we carefully examine the large zero submatrix that violates Property-$S$. Such a zero submatrix corresponds to a set of points and a set of lines such that no line is incident to any of the points. We argue in Lemma~\ref{le:not-propertys} that such a submatrix implies the existence of many ordinary lines. In fact, the conclusion is slightly more delicate: We either get many ordinary lines (in which case we are done) or we get a point with many ordinary lines incident to it (but not enough to complete the proof). In the second case, we need to perform an iterative argument which removes the point we found and applies the same argument again to the remaining points.

The second new ingredient in our proof comes into play only in the proof of Theorem~\ref{th:main}. Here, our goal is to improve on the rank bound of \cite{dsw14} using the existence of lines with four or more points. Recall that our goal is to give a good upper bound on the off-diagonal entries of $M = A^*A$. Consider the $(i,j)^{th}$ entry of $M$, obtained by taking the inner product of columns $i$ and $j$ in $A$. The $i^{th}$ column of $A$ contains the coefficients of $v_i$ in a set of  collinear triples containing $v_i$ (we might not use all collinear triples). In \cite{dsw14} this inner product is bounded using the Cauchy-Schwartz inequality, and uses the fact that we picked our triple family carefully so that $v_i$ and $v_j$ appear together in a small number of collinear triples. This does not use any information about possible cancellations that may occur in the inner product (considering different signs over the reals or angles of complex numbers). One of the key insights of our proof is to notice that having more than three points on a line, gives rise to  such cancellations (which increase the more points we have on a single line).  

To get a rough idea, let us focus on a set of points over the reals. Consider two points $v_1,v_2$ on a line that has two more points $v_3,v_4$ on it. Suppose that $v_3$ is `between' $v_1$ and $v_2$ and that $v_4$ is outside the interval $v_1,v_2$. Then, in the collinearity equation for the triple $v_1,v_2,v_3$ the signs of the coefficients of $v_1,v_2$ will both have the same sign. On the other hand, in the collinearity equation for $v_1,v_2,v_4$ the signs of the coefficients of $v_1,v_2$ will be different (one will be positive and the other negative). Thus, if both of these triples appear as rows of $A$, we will have non trivial cancellations! Of course, we need to also worry about the magnitudes of the coefficients but, luckily, this is possible since, if the coefficients are of magnitudes that differ from each other too much, we can ``win'' in another Cauchy-Schwartz (which again translates into a better rank bound, see Lemma~\ref{le:squares-bound}). To formalize the previous example, let $v_1, v_2, v_3, v_4$ be collinear points in $\RR^d$. Then there exist coefficients such that
\begin{align*}
& r \cdot v_1 + (1 - r) \cdot v_2 - v_3 = 0,\\
& s \cdot v_1 + (1 - s) \cdot v_2 - v_4  = 0,\\
\mbox{and} \quad & t \cdot v_1 + (1 - t) \cdot v_3 - v_4 = 0.
\end{align*}
Now at least one of $r(1-r)$, $s(1-s)$ and $t(1-t)$ must be negative, and at least one must be positive. Without loss of generality, say $r(1-r)$ is positive and $s(1-s)$ is negative. In order for the Cauchy-Schwarz inequality to be tight, we need that $r(1-r) = s(1-s),$ which cannot happen because one is positive and the other is negative.  This phenomena is captured in Lemma~\ref{le:coefficient-angle} which generalizes this idea to the complex numbers. The lemma only analyzes the case of four points since we can bootstrap the lemma for lines with more points by applying it to a random four tuple (see Item 4 of Lemma~\ref{le:dmatrix-line}).

%%%%%%%%%%%

\section{Preliminaries}\label{sec:prelims}

%%%%%%

\subsection{Matrix Scaling and Rank Bounds}

One of the main ingredients in our proof is rank bounds for design matrices. These techniques were first used for incidence type problems in \cite{bdwy13} and improved upon in \cite{dsw14}. We first set up some notation. For a complex matrix $A$, let $A^*$ denote the matrix conjugated and transposed. Let $A_{ij}$ denote the entry in the $i^{th}$ row and $j^{th}$ column of $A$,  and $R_k := R_k(A)$ and $C_i := C_i(A)$ denote the rows and column, respectively, of $A$. For two complex vectors $u, v \in \CC^d$, we denote their inner product by $\langle u, v\rangle = \sum_{i = 1}^{d} u_i \cdot \overline{v_i}$. The support of a vector $u \in \FF^n$ is defined to be $\supp(u) = \{i \in [n]: u_i \neq 0 \}$.

Central to obtaining the rank bounds for matrices is the notion of matrix scaling. We now introduce this notion and provide some definitions and lemmas.
\begin{definition}[Matrix Scaling]\label{def:matrixscaling}
Let $A$ be an $m \times n$ matrix over some field $\FF$. For every $\rho~\in~\FF^m$ and $\gamma~\in~\FF^n$ with all entries nonzero, the matrix $A^\prime$ with $A^\prime_{ij} = A_{ij}\cdot\rho_i\cdot\gamma_j$ is referred to as a {\em scaling} of $A$. We refer to the elements of the vectors $\rho$ and $\gamma$ as the {\em scaling coefficients}. Note that two matrices that are scalings of each other have the same rank and the same pattern of zero and non-zero entries.
\end{definition}

We will be interested in scalings of matrices that control the row and column sums. The following property provides a sufficient condition under which such scalings exist.
\begin{definition}[Property-$S$]\label{def:propertys}
Let $A$ be an $m \times n$ matrix over some field. We say that $A$ satisfies Property-$S$ if for every zero submatrix of size $a \times b$, we have \[ \frac{a}{m} + \frac{b}{n} \leq 1. \]
\end{definition}

The following theorem is given in  \cite{rs89}.
\begin{theorem}[Matrix Scaling theorem]
\label{th:realscaling}
Let $A$ be an $m \times n$ real matrix with non-negative entries satisfying Property-$S$. Then, for every $\epsilon > 0$, there exists a scaling $A^\prime$ of $A$ such that the sum of every row of $A^\prime$ is at most $1 + \epsilon$, and the sum of every column of $A^\prime$ is at least $m/n - \epsilon$. Moreover, the scaling coefficients are all positive real numbers.
\end{theorem}
In fact, in the above theorem, we may assume that the sum of every row of $A^\prime$ is exactly $1 + \epsilon$, since scaling the rows to make the row sums $1 + \epsilon$ can only increase the column sums.

The following Corollary to Theorem~\ref{th:realscaling} appeared in \cite{bdwy13}.
\begin{corollary}[$\ell_2$ scaling]
\label{co:complexscaling}
Let $A$ be an $m \times n$ complex matrix satisfying Property-$S$. Then, for every $\epsilon > 0$, there exists a scaling $A^\prime$ of $A$ such that for every $i \in [m]$
\[ \sum_{j \in [n]} \left|A^\prime_{ij}\right|^2 \leq 1 + \epsilon, \]
and for every $j \in [n]$
\[ \sum_{i \in [m]} \left|A^\prime_{ij}\right|^2 \geq \frac{m}{n} - \epsilon. \]
Moreover, the scaling coefficients are all positive real numbers.
\end{corollary}
Corollary~\ref{co:complexscaling} is obtained by applying Theorem~\ref{th:realscaling} to the matrix obtained by squaring the absolute values of the entries of the matrix $A$. Once again, we may assume that $\sum_{j \in [n]} \left|A^\prime_{ij}\right|^2 = 1 + \epsilon$.

To bound the rank of a matrix $A$, we will bound the rank of the matrix $M = A^{\prime*}A^\prime$, where $A^\prime$ is some scaling of $A$. Then we have that $\rank(A) = \rank(A^\prime) = \rank(M)$. We use Corollary~\ref{co:complexscaling}, along with rank bounds for diagonal dominant matrices. The following lemma is a variant of a folklore lemma on the rank of diagonal dominant matrices (see~\cite{alon09}) and appeared in this form in~\cite{dsw14}. 
\begin{lemma}
\label{le:rankbound}
Let $A$ be an $n \times n$ complex hermitian matrix, such that $|A_{ii}| \geq L$ for all $i \in n$. Then
\[ \rank(A) \geq \frac{n^2L^2}{nL^2 + \sum_{i \neq j}|A_{ij}|^2}. \]
\end{lemma}

The matrix scaling theorem allows us to control the $\ell_2$ norms of the columns and rows of $A'$, which in turn allow us to bound the sums of squares of entries of $M$. For this, we use a variation of a lemma from \cite{dsw14}. While the proof idea is the same, our proof requires a somewhat more careful analysis.

For any $x_1, x_2, \dots, x_n \in \RR$, a standard application of the Cauchy-Schwarz inequality implies
\begin{equation}
\label{eq:csinequality}
\left(\sum_{i = 1}^{n} x_i\right)^2 \leq n \sum_{i = 1}^{n}x_i^2,
\end{equation}
with equality if and only if $x_1 = x_2 = \dots = x_n$. As mentioned in Section~\ref{sec:overview}, a key ingredient in our proof is to use information about cancellations in the inner product. To this end, we make use of the following:
\begin{equation}
\label{eq:csequality}
\left(\sum_{i = 1}^{n} x_i\right)^2 + \sum_{i < j} (x_i - x_j)^2 = n \sum_{i = 1}^{n}x_i^2.
\end{equation}
The latter term on the LHS in Equation~\eqref{eq:csequality} quantifies the difference between the terms on two sides of Equation~\eqref{eq:csinequality}. The proof for Lemma~\ref{le:offdiagonalsum} will use Equation~\eqref{eq:csequality} twice. To simplify the presentation, we need the following definitions: 
\begin{definition}\label{def:gapparameters}
Let $A$ be an $m \times n$ matrix over $\CC$, and let $\supp(i, j) := \supp(C_i) \cap \supp(C_j)$. Then we define: 
\[D(A) := \sum_{i \neq j} \sum_{\substack{{k < k^\prime}\\{k, k' \in \supp(i, j)}}} \left| A_{ki}\overline{A_{kj}} - A_{k^\prime i}\overline{A_{k^\prime j}} \right|^2, \]
and\[ E(A) := \sum_{k = 1}^{m} \sum_{\substack{{i < j}\\{i, j \in \supp(R_k)}}} \left(|A_{ki}|^2 - |A_{kj}|^2\right)^2. \]
\end{definition}
Note that both $D(A)$ and $E(A)$ are non-negative real numbers.
\paragraph{Discussion on $D(A)$ and $E(A)$:}
To give some insight into the proof, we include a brief discussion about $D(A)$ and $E(A)$. Consider the matrix
$$A = \begin{bmatrix}
1  &   1  &   2   &  0\\
1  &   1  &   0   &  1\\
1  &   0  &   1   &  1\\
0  &   1  &   1   &  1\\
\end{bmatrix}, $$ and let $M = A^* A$. Observe that $M_{ij}^2 = \langle C_i, C_j \rangle^2 = \left( \sum_{k = 1}^{4} A_{ki} A_{kj} \right)^2$. Equation~\eqref{eq:csinequality} implies that $\left( \sum_{k = 1}^{4} A_{ki} A_{kj} \right)^2 \leq 4 \sum_{k = 1}^{4} A_{ki}^2 A_{kj}^2.$ An obvious improvement here comes from the fact the supports of any two columns intersect in only two entries, giving $\left( \sum_{k = 1}^{4} A_{ki} A_{kj} \right)^2 \leq 2 \sum_{k = 1}^{4} A_{ki}^2 A_{kj}^2$. For~$i = 1$, $j = 2$, this gives $(1 + 1)^2 = 4 \leq 2 (1^2 + 1^2) = 4$ and is tight because $A_{11} A_{12} = A_{21} A_{22} = 1$. For $i = 1$, $j = 3$, it implies $(2 + 1)^2 = 9 \leq 2(2^2 + 1^2) = 10$. The difference is precisely captured by $(A_{11} A_{13} - A_{21} A_{23})^2 = (2 - 1)^2 = D(A)$.

$E(A)$ comes into the picture in bounding $\sum_{i = 1}^{4} A_{ki}^4$, for fixed $k$. Equation~\eqref{eq:csinequality} implies $\sum_{i = 1}^{4} A_{ki}^4 \geq \frac{1}{3}\left(\sum_{i = 1}^{4} A_{ki}^2\right)^2$. Equality holds if and only if all non-zero entries in row $k$ are equal., e.g., when $k \in {2, 3, 4}$. Otherwise, e.g., when $k = 1$, $E(A)$ quantifies the difference.

In Section~\ref{sec:dependencymatrix}, we will show that for the dependency matrix, each row will either have all non-zero entries far from each other, i.e., it will contribute to $E(A)$, or it will, together with other rows, contribute to $D(A)$.

We now give the main result for this section.
%%%%%% \sum M_{ij}
\begin{lemma}
\label{le:offdiagonalsum}
Let $A$ be an $m \times n$ matrix over $\CC$. Suppose that each row of $A$ has $\ell_2$ norm  $< \alpha$, the supports of every two columns of $A$ intersect in at most $t$ locations, and the size of the support of every row is at most $q$. Let $M = A^*A$.
Then \[ \sum_{i \neq j} |M_{ij}|^2 \leq \left( 1 - \frac{1}{q}\right) tm\alpha^4 - \left(D(A) + \frac{t}{q}E(A)\right). \]
\end{lemma}

\begin{proof}
Note that 
\begin{equation}
\label{eq:od1}
 \sum_{i \neq j} |M_{ij}|^2 = \sum_{i \neq j} |\langle C_i, C_j \rangle|^2  =  \sum_{i \neq j} \left| \sum_{k = 1}^{m} A_{ki} \overline{A_{kj}} \right|^2 =  \sum_{i \neq j} \left| \sum_{k \in \supp(i, j)} A_{ki} \overline{A_{kj}} \right|^2.
\end{equation}
Here the last equality uses the fact that the only terms contributing to the sum in the inner summation come from $k~\in~\supp(i, j)$. Since the supports of any two columns of $A$ intersect in at most~$t$ locations, $|\supp(i, j)| \leq t$ for all $i \neq j$. Combining this with Equation~\ref{eq:csequality} gives:
\begin{align*}
& \sum_{i \neq j} \left| \sum_{k \in \supp(i, j)} A_{ki} \overline{A_{kj}} \right|^2 \\
& \leq \sum_{i \neq j} \left (t \sum_{k \in \supp(i, j)} |A_{ki}|^2|A_{kj}|^2 - \sum_{\substack{{k < k^\prime}\\{k, k' \in \supp(i, j)}}} \left| A_{ki}\overline{A_{kj}} - A_{k^\prime i}\overline{A_{k^\prime j}} \right|^2 \right)\\
& = \sum_{i \neq j} \left (t \sum_{k = 1}^{m} |A_{ki}|^2|A_{kj}|^2 - \sum_{\substack{{k < k^\prime}\\{k, k' \in \supp(i, j)}}} \left| A_{ki}\overline{A_{kj}} - A_{k^\prime i}\overline{A_{k^\prime j}} \right|^2 \right) \\
& = t\sum_{i \neq j} \sum_{k = 1}^{m} |A_{ki}|^2|A_{kj}|^2 - D(A) \\
& = t \sum_{k = 1}^{m} \left(\sum_{i=1}^{n} |A_{ki}|^2 \right)^2 - t \sum_{k = 1}^{m} \left(\sum_{i=1}^{n} |A_{ki}|^4 \right) -  D(A).
\end{align*}
We now use Equation~\ref{eq:csequality} again to bound the second term in the expression above. This time we may restrict our attention to the at most  $q$ non-zero entries in each row. This gives
\begin{align*}
\sum_{i \neq j} |M_{ij}|^2 & \leq t \sum_{k = 1}^{m} \left(\sum_{i=1}^{n} |A_{ki}|^2 \right)^2 - t \sum_{k = 1}^{m} \frac{1}{q}  \left( \left( \sum_{i = 1}^{n} |A_{ki}|^2 \right)^2 + \sum_{\substack{{i < j}\\{i, j \in \supp(R_k)}}} \left(|A_{ki}|^2 - |A_{kj}|^2 \right)^2 \right) -  D(A)\\ 
& = \left(1 - \frac{1}{q}\right) t \sum_{k = 1}^{m} \left(\sum_{i=1}^{n} |A_{ki}|^2 \right)^2  - \frac{t}{q} \sum_{k = 1}^{m} \sum_{\substack{{i < j}\\{i, j \in \supp(R_k)}}} \left(|A_{ki}|^2 - |A_{kj}|^2\right)^2 -  D(A) \\
& = \left(1 - \frac{1}{q}\right) t \sum_{k = 1}^{m} \left(\sum_{i=1}^{n} |A_{ki}|^2 \right)^2   - \frac{t}{q}E(A) -  D(A)  \\
& = \left(1 - \frac{1}{q}\right)tm\alpha^4 - \left(D(A) + \frac{t}{q}E(A)\right).
\end{align*}
\end{proof}

From this, we get the following easy corollary.

\begin{corollary}
\label{co:offdiagonalsumineq}
Let $A$ be an $m \times n$ matrix over $\CC$. Suppose that each row of $A$ has $\ell_2$ norm $ < \alpha$, the supports of every two columns of $A$ intersect in at most $t$ locations, and the size of the support of every row is at most $q$. Let $M = A^*A$.
Then \[ \sum_{i \neq j} |M_{ij}|^2 \leq \left( 1 - \frac{1}{q}\right) tm\alpha^4.  \]
\end{corollary}

%%%%%

\subsection{Latin squares}

Latin squares play a central role in our proof. While Latin squares play a role in both \cite{dsw14} and \cite{bdwy13}, our proof exploits their design properties more strongly.

\begin{definition}[Latin square]
An $r \times r$ Latin square is an $r \times r$ matrix $L$ such that $L_{ij} \in [r]$ for all $i, j$ and every number in $[r]$ appears exactly once in each row and exactly once in each column.
\end{definition}
If $L$ is a Latin square and $L_{ii} = i$ for all $i \in [r]$, we call it a {\em diagonal} Latin square. 

\begin{lemma}\label{le:latinsq}
For every $r \geq 3$, there exists an $r \times r$ diagonal Latin square. For $r \geq 4$, there exist diagonal Latin squares with the additional property that, for every $i \neq j$, $L_{ij} \neq L_{ji}$.
\end{lemma}
\begin{proof}
For $r \geq 3$, the existence of $r \times r$ diagonal Latin squares was proved by Hilton~\cite{hil73}. Therefore, we need only show the second part of the theorem. For this we rely on {\em self-orthogonal Latin squares}.

Two Latin squares $L$ and $L^\prime$ are called {\em orthogonal} if every ordered pair $(k, l) \in [r]^2$ occurs uniquely as $(L_{ij}, L^\prime_{ij})$ for some $i, j \in [r]$. A Latin square is called {\em self-orthogonal} if it is orthogonal to its transpose, denoted by $L^T$. A theorem of Brayton, Coppersmith, and Hoffman~\cite{bch74} proves the existence of $r \times r$ self-orthogonal Latin squares for $r \in \NN$, $r \neq 2, 3, 6$. Let $L$ be a self-orthogonal Latin square. Since $L_{ii} = L^T_{ii}$, the diagonal entries give all pairs of the form $(i, i)$ for every $i \in [r]$, i.e., the diagonal entries must be a permutation of $[r]$. Without loss of generality, we may assume that $L_{ii} = i$ and so $L$ is also a diagonal Latin square. Clearly a self-orthogonal Latin square satisfies the property that $L_{ij} \neq L_{ji}$ if $i \neq j$.

This leaves us only with the case $r = 6$, which requires separate treatment. It is known that $6\times 6$ self-orthogonal Latin squares do not exist. Fortunately, the property we require is weaker and we are able to give an explicit construction of a matrix that is sufficient for our needs. Let $L$ be the following matrix
$$\begin{bmatrix}
1  &   4  &   5   &  3  &   6  &   2\\
3  &   2  &   6   &  5  &   1  &   4\\
2  &   5  &   3   &  6  &   4  &   1\\
6  &   1  &   2   &  4  &   3  &   5\\
4  &   6  &   1   &  2  &   5  &   3\\
5  &   3  &   4   &  1  &   2  &   6\\
\end{bmatrix}. $$
It is straightforward to verify that $L$ has the required properties.
\end{proof}

The following lemma is a strengthening of a lemma from \cite{bdwy13}.
\begin{lemma}
\label{le:triplesystem}
Let $r \geq 3$. Then there exists a set $T \subseteq [r]^3$, referred to as a {\em triple system}, of $r^2 - r$ triples that satisfies the following properties:
\begin{enumerate}
\item Each triple consists of three distinct elements;
\item For every pair $i, j \in [r]$, $i \neq j$, there are exactly six triples containing both $i$ and $j$;
\item If $r \geq 4$, for every $i, j \in [r]$, $i \neq j$, there are at least two triples containing $i$ and $j$ such that the remaining elements are distinct.
\end{enumerate}
\end{lemma}
\begin{proof}
Let $L$ be a Latin square as in Lemma~\ref{le:latinsq}. Let $T$ be the set of triples $(i, j, k) \subseteq [r]^3$ with $i \neq j$ and $k = L_{ij}$. Clearly the number of such triples is $r^2 - r$. We verify that the properties mentioned hold. 

Recall that we have $L_{ii} = i$ for all $i \in [r]$, and every value appears once in each row and column. So for $i \neq j \in [r]$, it can not happen that $L_{ij} = i$ or $L_{ij} = j$ and we get Property 1, i.e., all elements of a triple must be distinct.

For Property 2, note that a pair $i, j$ appears once as $(i, j, L_{ij})$ and once as $(j, i, L_{ji})$. And since every element appears exactly once in every row and column, we have that $i$ must appear once in the $j^{th}$ row, $j$ must appear once in the $i^{th}$ row and the same for the columns. It follows that each of $(*, j, i), (j,*, i), (*, i, j)$ and $(i, *, j)$ appears exactly once, where $*$ is some other element of $[r]$. This gives us that every pair appears in exactly six triples.

For $r \geq 4$ and $i \neq j$, since $L_{ij} \neq L_{ji}$, the triples $(i, j, L_{ij})$ and $(j, i, L_{ji})$ are sufficient to satisfy Property 3. 
\end{proof}

\section{The dependency matrix}
\label{sec:dependencymatrixprelim}

Let $\V = \{v_1, \dots, v_n\}$ be a set of $n$ points in $\CC^d$.
We will use $\dim(\V)$ to denote the dimension of the linear span of $\V$, and by $\adim(\V)$ the dimension of the affine span of $\V$ (i.e., the minimum $r$ such that points of $\V$ are contained in a shift of a linear subspace of dimension $r$).
We projectivize $\CC^d$ and consider the set of vectors $\V^\prime = \{v^\prime_1, \dots, v^\prime_n\}$, where $v^\prime_i = (v_i, 1)$ is the vector in $\CC^{d+1}$ obtained by appending a 1 to the vector $v_i$. Let $V$ be the $n \times (d+1)$ matrix whose $i^{th}$ row is the vector $v^\prime_i$. Now note that 
\[ \adim(\V) = \dim(\V^\prime) - 1 = \rank(V) - 1. \]

We now construct a matrix $A$, which we refer to as the dependency matrix of $\V$. The construction we give here is preliminary, but suffices to prove Theorems~\ref{th:3n/2} and \ref{th:higherdim}. A refined construction is given in Section~\ref{sec:dependencymatrix}, where we select the triples more carefully. The rows of the matrix will consist of linear dependency coefficients, which we define below.
\begin{definition}[Linear dependency coefficients]
Let $v_1, v_2$ and $v_3$ be three distinct collinear points in $\CC^d$, and let $v_i^\prime = (v_i, 1)$, $i \in \{1, 2, 3\}$, be vectors in $\CC^{d+1}$. Recall that $v_1, v_2, v_3$ are collinear if and only if there exist nonzero coefficients $a_1, a_2 ,a_3\in \CC$ such that \[ a_1 v^\prime_1 + a_2 v^\prime_2 + a_3v^\prime_3 = 0. \]
We refer to the $a_1, a_2$ and $a_3$ as the linear dependency coefficients between $v_1, v_2, v_3$. Since these coefficients are determined up to scaling by a complex number, we fix a canonical choice by setting $a_3 = 1$.
\end{definition}

\begin{definition}[Dependency Matrix]\label{def:dependency}
For every line $l \in \L_{\geq 3}(\V)$, let $\V_l$ denote the points lying on $l$. Then $|\V_l| \geq 3$ and we assign each line a triple system $T_l \subseteq [|\V_l|]^3$, the existence of which is guaranteed by  Lemma~\ref{le:triplesystem}. Let $A$ be the $m \times n$ matrix obtained by going over every line $l \in \L_{\geq 3}$ and for each triple $(i, j, k) \in T_l$, adding as a row of $A$ a vector with three nonzero coefficients in positions $i, j, k$ corresponding to the linear dependency coefficients among the points $v_i, v_j, v_k$. We refer to $A$ as the dependency matrix for $\V$.
\end{definition}

Note that we have $AV = 0$. Every row of $A$ has exactly three nonzero entries. By Property 2 of Lemma~\ref{le:triplesystem}, the supports of any distinct two columns intersect in exactly six entries when the two corresponding points lie on a special line, and are disjoint otherwise. That is, the supports of any two distinct columns intersect in at most six entries.

We say a pair of points $v_i, v_j$, $i \neq j$, {\em appears} in the dependency matrix $A$ if there exists a row with nonzero entries in columns $i$ and $j$. The number of times a pair appears is the number of rows with nonzero entries in both columns $i$ and $j$.

Every pair of points that lies on a special line appears exactly six times. The only pairs not appearing in the matrix are pairs of points that determine ordinary lines. There are $n \choose 2$ pairs of points, $t_2(\V)$ of which determine ordinary lines. So the number of pairs appearing in $A$ is ${n \choose 2} - t_2$. The total number of times these pairs appear is then $6\left({n \choose 2} - t_2\right)$. Every row gives three distinct pairs of points, so it follows that the number of rows of $A$ is
\begin{align}\label{eq:num-rows}
m = 6\left({n \choose 2} - t_2\right)/3 = n^2 - n - 2t_2(\V).
\end{align}
Note that $m > 0$, unless $t_2 = {n \choose 2}$, i.e., all lines are ordinary.

As mentioned in the proof overview, we will consider two cases: when $A$ satisfies Property-$S$ and when it does not. We now prove lemmas dealing with the two cases. The following lemma deals with the former case.

\begin{lemma} \label{th:3n/2 Prop S}
Let $\V$ be a set of $n$ points affinely spanning $\CC^d$, $d \geq 3$, and let $A$ be the dependency matrix for $\V$. Suppose that $A$ satisfies Property-$S$. Then \[ t_2(\V) \geq \frac{(d-3)}{2(d+1)} n^2 + \frac{3}{2} n. \]
\end{lemma}

\begin{proof}
Fix $\epsilon>0.$ Since $A$ satisfies Property-$S$, by Corollary~\ref{co:complexscaling} there is a scaling $A'$ such that the $\ell_2$ norm of each row is at most $\sqrt{1+\epsilon}$ and the $\ell_2$ norm of each column is at least $\sqrt{\frac{m}{n}-\epsilon}$.  Let $M:=A'^*A'.$  Then
$M_{ii}\geq \frac{m}{n}-\epsilon$ for all $i$.  Since every row in $A$ has support of size three, and the supports of any two columns intersect in at most six locations,  Corollary \ref{co:offdiagonalsumineq} gives us that $\sum\limits_{i\neq j}|M_{ij}|^2 \leq 4m(1+\epsilon)^2.$  Applying Lemma~\ref{le:rankbound} to $M$, we have 
\[ \rank(M)\geq \frac{n^2(\frac{m}{n}-\epsilon)^2}{n(\frac{m}{n}-\epsilon)^2+ 4m(1+\epsilon)^2}. \]  Taking $\epsilon$ to 0, and combining with Equation \eqref{eq:num-rows} gives
\begin{align*}
\rank(A) = \rank(A') = \rank(M) & \geq \frac{n^2\frac{m^2}{n^2}}{n\frac{m^2}{n^2}+4m} = \frac{mn}{m+4n}\\
& = n-\frac{4n^2}{m+4n} = n-\frac{4n^2}{n^2-n-2t_2(\V) +4n}\\
& = n-\frac{4n^2}{n^2+3n-2t_2(\V)}.
\end{align*}
Recall that $\adim(\V) = d = \rank(V) - 1$. Since $AV = 0$, we have $\rank(V) \leq n - \rank(A)$. It follows that
\begin{align*}
d + 1 & \leq \frac{4n^2}{n^2+3n-2t_2(\V)}, \\
\mbox{i.e.,} \quad t_2(\V) & \geq \frac{(d-3)}{2(d+1)} n^2 + \frac{3}{2} n.
\end{align*}
\end{proof}

We now consider the case when Property-$S$ is not satisfied.
\begin{lemma}\label{le:not-propertys}
Let $\V$ be a set of $n$ points in $\CC^d$, and let $A$ be the dependency matrix for $\V$. Suppose that $A$ does not satisfy Property-$S$. Then, for every integer $b^*$, $1 < b^* < 2n/3$, at least one of the following holds:
\begin{enumerate}
\item There exists a point $v \in \V$ contained in at least $\frac{2}{3}(n+1) - b^*$ ordinary lines;
\item $ t_2(\V) \geq nb^*/2$.
\end{enumerate}
\end{lemma}

\begin{proof}
Since $A$ violates Property-$S$, there exists a zero submatrix supported on rows $U \subseteq [m]$ and columns $W \subseteq [n]$ of the matrix $A$, where $|U| = a$ and $|W| = b$, such that \[ \frac{a}{m} + \frac{b}{n} > 1.\]  Let $X = [m] \setminus U$ and $Y = [n] \setminus W$ and note that $|X| = m - a$ and $|Y| = n - b$.  Let the violating columns correspond to the set $\V_1 = \{v_1, \dots,v_b\} \subset \V$.  We consider two cases: when $b < b^*$, and when $b \geq b^*$.

{\bf Case 1} $(b <b^*)$. We may assume that $U$ is maximal, so every row in the submatrix $X \times W$ has at least one nonzero entry. Partition the rows of $X$ into three parts: Let $X_1, X_2$ and $X_3$ be rows with one, two and three nonzero entries in columns of $W$ respectively.  We will get a lower bound on the number of ordinary lines containing exactly one point in $\V_1$ and one point in $\V \setminus \V_1$ by bounding the number of pairs $\{v_i, w\}$, with $v_i \in \V_1$ and $w \in \V \setminus \V_1$, that lie on special lines. Note that there are at most $b(n-b)$ such pairs, and each pair that does not lie on a special line determines an ordinary line.

Each row of $X_1$ gives two pairs of points $\{v_i, w_1\}$ and $\{v_i, w_2\}$ that lie on a special line, where $v_i \in \V_1$ and $w_1,w_2 \in \V\setminus \V_1$. Each row of $X_2$ gives two pairs of points $\{v_i, w\}$ and $\{v_j, w\}$, where $v_i,v_j \in \V_1$ and $w \in\V\setminus\V_1$ that lie on special lines. Each row of $X_3$ has all zero entries in the submatrix supported on $X\times Y$, so does not contribute any pairs. Recall that each pair of points on a special line appears exactly six times in the matrix.  This implies that the number of pairs that lie on special lines with at least one point in $\V_1$ and one point in $\V \setminus \V_1$ is $\frac{2|X_1|+2|X_2|}{6} \leq \frac{2|X|}{6}$. Hence, the number of ordinary lines containing exactly one of $v_1, \dots ,v_b$ is at least $b(n-b) - \frac{|X|}{3}$. 

Recall that
\begin{align*}
 1 < \frac{a}{m} + \frac{b}{n} = \left(1 - \frac{|X|}{m}\right) + \frac{b}{n}.
\end{align*}
Substituting $m \leq n^2 - n$ (a consequence of Equation \eqref{eq:num-rows}) gives 
\begin{align*}
|X| & < \frac{bm}{n} \leq b(n - 1).
\end{align*}
This gives that the number of ordinary lines containing exactly one point in $\V_1$ is at least
 \[ b(n-b) - \frac{|X|}{3}  > \frac{2b}{3}n - \frac{3b^2-b}{3}. \]
We now have that there exists $v \in \V_1$ such that the number of ordinary lines containing $v$ is at least
\[ \left\lfloor\frac{2}{3}n- \frac{3b-1}{3}\right\rfloor \geq \left\lfloor \frac{2}{3}n - b^* + \frac{4}{3}\right\rfloor \geq \frac{2}{3}(n+1) - b^*. \]

{\bf Case 2} $(b \geq b^*)$. We will determine a lower bound for $t_2(\V)$ by counting the number of nonzero pairs of entries $A_{ij}, A_{ij'}$ with $j \neq j^\prime$, that appear in the submatrix $U \times Y$.  There are ${n - b\choose 2}$ pairs of points in $\V \setminus \V_1$, each of which appears at most six times, therefore the number of pairs of such entries is at most $6{n-b\choose 2}$.  Each row of $U$ has three pairs of nonzero entries, i.e., the number of pairs of entries equals $3a$. It follows that 
\begin{align}
\label{eq:coefficientpairs3}3a \leq 6{n-b \choose 2}.
\end{align}
Combining Equation \eqref{eq:num-rows} with the fact that $\frac{a}{m} +\frac{b}{n} > 1$ implies
\begin{align}
\label{eq:coefficientpairs4}
a > m\left(1-\frac{b}{n}\right) = \left(n^2-n-2t_2(\V)\right)\left(1-\frac{b}{n}\right).
\end{align}
Equations~(\ref{eq:coefficientpairs3}) and (\ref{eq:coefficientpairs4}) together give 
\begin{align*}
\left(n^2-n-2t_2(\V)\right)\left(1-\frac{b}{n}\right) < 2{n-b\choose 2}.
\end{align*}
Finally, solving for $t_2(\V)$, we have
\[ t_2(\V) > \frac{nb}{2}  \geq \frac{nb^*}{2}. \]
\end{proof}

\section{Proofs of Theorems~\ref{th:3n/2} and \ref{th:higherdim}}
\label{sec:3n/2}

The proofs of both Theorems~\ref{th:3n/2} and \ref{th:higherdim} rely on Lemmas~\ref{th:3n/2 Prop S} and \ref{le:not-propertys}. Together, these lemmas imply that there must be a point with many ordinary lines containing it, or there are many ordinary lines in total. As mentioned in the proof overview, the theorems are then obtained by using an iterative argument removing a point contained in many ordinary lines, and then applying the same argument to the remaining points.

\subsection{Proof of Theorem \ref{th:3n/2}}

We get the following easy corollary from Lemma~\ref{th:3n/2 Prop S} and Lemma~\ref{le:not-propertys}.
\begin{corollary}\label{th:eitheror}
Let $\V$ be a set of $n \geq 5$ points in $\CC^d$ not contained in a plane.
Then at least one of the following holds:
\begin{enumerate}
\item There exists a point $v \in \V$ contained in at least $\frac{2}{3}n - \frac{7}{3}$ ordinary lines;
\item $t_2(\V) \geq \frac{3}{2}n$.
\end{enumerate}
\end{corollary}
\begin{proof}
Let $A$ be the dependency matrix for $\V$. If $A$ satisfies Property-$S$, then we are done by Lemma~\ref{th:3n/2 Prop S}. Otherwise, let $b^* = 3$, and note that Lemma~\ref{le:not-propertys} gives us the statement of the corollary when $n \geq 5$.
\end{proof}

We are now ready to prove Theorem~\ref{th:3n/2}. For convenience, we state the theorem again.
\begin{reptheorem}{th:3n/2}
Let $\V$ be a set of $n \geq 24$ points in $\CC^3$ not contained in a plane. Then $\V$ determines at least $\frac{3}{2}n$ ordinary lines, unless $n-1$ points are on a plane in which case there are at least $n - 1$ ordinary lines.
\end{reptheorem}

\begin{proof}
If $t_2(\V) \geq \frac{3}{2}n$ then we are done. Else, by Corollary~\ref{th:eitheror}, there exists a point $v_1$ incident to at least $\frac{1}{3}(2n - 7)$ ordinary lines and hence, at most $\frac{1}{6}(n + 4)$ special lines. Let $\V_1 = \V \setminus \{ v_1 \}$. If $\V_1$ is planar, then there are exactly $n - 1$ ordinary lines incident to $v_1$. This is the only case where there exists fewer than $\frac{3}{2}n$ ordinary lines.

Suppose now that $\V_1$ is not planar. Again, by Corollary~\ref{th:eitheror}, there are either $\frac{3}{2}(n - 1)$ ordinary lines determined by $\V_1$, or there exists a point $v_2 \in \V_1$ incident to at least $\frac{2}{3}(n - 1) - \frac{7}{3} = \frac{1}{3}(2n - 9)$ ordinary lines (determined by $\V_1$). In the former case, there are $\frac{3}{2}(n - 1)$ ordinary lines determined by $\V_1$, at most $\frac{1}{6}(n + 4)$ of which could contain $v_1$. Then, the total number of ordinary lines in $\V$ is
\[ t_2(\V) \geq \frac{3}{2}(n - 1) -  \frac{1}{6}(n + 4) + \frac{1}{3}(2n - 7) = \frac{1}{2}(4n - 9). \]
When $n \geq 9$, this implies $t_2(\V) \geq \frac{3}{2}n$.

In the latter case there exists a point $v_2 \in \V_1$ with at least $\frac{1}{3}(2n - 9)$ ordinary lines determined by $\V_1$ incident to it. At most one of these could contain $v_1$, so there are at least $\frac{1}{3}(2n - 7) + \frac{1}{3}(2n - 9) - 1 = \frac{1}{3}(4n - 19)$ ordinary lines incident to one of  $v_1$ or $v_2$. Note also that the number of special lines containing one of $v_1$ or $v_2$ is at most $\frac{1}{6} (n + 4) + \frac{1}{6}(n + 3) = \frac{1}{6}(2n + 7)$.

Let $\V_2 = \V_1 \setminus \{v_2\}$. If $\V_2$ is contained in a plane, there are at least $n - 3$ ordinary lines incident to each of $v_1$ and $v_2$ giving a total of $2n - 6$ ordinary lines determined by $\V$. It follows that when $n \geq 12$, $t_2(\V) \geq \frac{3}{2}n.$

Otherwise, $\V_2$ is not contained in a plane, and again Corollary~\ref{th:eitheror} gives two cases. If there are $\frac{3}{2}(n - 2)$ ordinary lines determined by $\V_2$, then the total number of ordinary lines is
\[ t_2(\V) \geq \frac{3}{2} (n - 2) - \frac{1}{6}(2n + 7) + \frac{1}{3}(4n - 19) = \frac{1}{2} (5n - 21). \]
When $n \geq 11$, this implies $t_2(\V) \geq \frac{3}{2}n$.

Otherwise, there exists a point $v_3$ contained in at least $\frac{2}{3}(n - 2) - \frac{7}{3}$ ordinary lines. At most two of these could contain $v_1$ or $v_2$, so there are $\frac{2}{3}(n - 2) - \frac{7}{3} - 2 = \frac{1}{3}(2n - 17)$ ordinary lines incident to $v_3$ determined by $\V$. Summing up the number of lines containing one of  $v_1, v_2$ and $v_3$, we have
\[ t_2(\V) \geq \frac{1}{3}(2n-17) + \frac{1}{3}(4n - 19) = 2n - 12. \]
When $n \geq 24$, this implies $t_2(\V) \geq \frac{3}{2}n$.
\end{proof}

\subsection{Proof of Theorem \ref{th:higherdim}}\label{sec:higherdim}

We get the following easy corollary from Lemma~\ref{th:3n/2 Prop S} and Lemma~\ref{le:not-propertys}.

\begin{corollary}\label{th:eitherorgap1.8}
There exists a positive integer $n_0$ such that the following holds. Let $\V$ be a set of $n \geq n_0$ points in $\CC^d$ not contained in a three dimensional affine subspace. Then at least one of the following holds:
\begin{enumerate}
\item There exists a point contained in at least $\frac{n}{2}$ ordinary lines;
\item $ t_2(\V) \geq \frac{1}{12}n^2$.
\end{enumerate}
\end{corollary}
\begin{proof}
Let $A$ be the dependency matrix of $\V$. If $A$ satisfies Property-$S$, then we are done by Lemma~\ref{th:3n/2 Prop S}. Otherwise, let $b^* = n/6$. Now by Lemma~\ref{le:not-propertys}, either the number of ordinary lines
\[t_2(\V) \geq \frac{n}{2}b^* \geq \frac{1}{12}n^2, \]
or, there exists a point $v \in \V$, such that the number of ordinary lines containing $v$ is at least
\[ \frac{2}{3}(n+1) - b^* > \frac{1}{2}n. \]
\end{proof}

We are now ready to prove Theorem~\ref{th:higherdim}. For convenience, we state the theorem again.
\begin{reptheorem}{th:higherdim}
There exists an absolute constant $c^\prime > 0$ and a positive integer $n_0$ such that the following holds. Let  $\V$ be a set of $n \geq n_0$ points in $\CC^4$ with at most $\frac{1}{2} n$ points contained in any three dimensional affine subspace. Then
\[ t_2(\V) \geq c^\prime n^2. \]
\end{reptheorem}
\begin{proof}
The basic idea of the proof uses the following algorithm: We use Corollary~\ref{th:eitherorgap1.8} to find a point incident to a large number of ordinary lines,  ``prune'' this point, and then repeat this on the smaller set of points. We stop when either we can not find such a point, in which case Corollary~\ref{th:eitherorgap1.8} guarantees a large number of ordinary lines, or when we have accumulated enough ordinary lines.

Consider the following algorithm:

Let $\V_0 := \V$ and $j = 0$.
\begin{enumerate}
\item If $\V_j$ satisfies case (2) of Corollary~\ref{th:eitherorgap1.8}, then stop.
\item Otherwise, there must exist a point $v_{j+1} \in \V_j$ incident to at least $\frac{n - j}{2}$ ordinary lines determined by $\V_j$. Let $\V_{j+1} = \V_j \setminus \{ v_{j+1} \}$.
\item Set $j = j + 1$. If $j \geq n/2$, then stop. Otherwise, go to Step 1.
\end{enumerate}

Since no three dimensional subspace contains more than $n/2$ points, at no point will the algorithm stop because the configuration becomes three dimensional. That is, we can use Corollary~\ref{th:eitherorgap1.8} at every step of the algorithm.

We now analyze the two stopping conditions for the algorithm, and show that we can always find enough ordinary lines by the time the algorithm stops.

Suppose that we stop because $\V_j$ satisfies case (2) of Corollary~\ref{th:eitherorgap1.8} for some $1 \leq j < n/2$.  Case (2) of Corollary~\ref{th:eitherorgap1.8}  implies 
\begin{align}
\label{eq:ordinaryvj1.8} t_2(\V_j)\geq \frac{(n-j)^2}{12}.
\end{align}
On the other hand, each pruned point $v_i$, $1 \leq i \leq j$, is incident to at least $\frac{n - i + 1}{2} > \frac{n - i}{2}$ ordinary lines determined by $\V_{i-1}$, and hence, at most $(n - i - \frac{n - i + 1}{2})/2 < \frac{n - i}{4}$ special lines. An ordinary line in $\V_i$ might not be ordinary in $\V_{i-1}$ if it contains $v_{i}$. Thus, in order to lower bound the total number of ordinary lines in $\V$,
we sum over the number of ordinary lines contributed by each of the pruned points $v_{i}$,
$1 \leq i \leq j$, and subtract from the count the number of potential lines that could contain $v_{i}$.  Then the number of ordinary lines in $\V$ contributed by the pruned points is at least
\begin{align}
\label{eq:ordinarypruning1.8}  \sum_{i=1}^j \left(\frac{n - i}{2} - \frac{n - i}{4} \right) = \frac{1}{4}\sum_{i=1}^j \left(n - i\right).
\end{align}
Combining (\ref{eq:ordinaryvj1.8}) and (\ref{eq:ordinarypruning1.8}) gives 
\begin{align*}
t_2(\V) & \geq \frac{1}{12}(n-j)^2 + \frac{1}{4}\sum_{i=1}^j \left(n - i\right) \\
& = \frac{1}{24}\left(- j^2 + j(2n - 3) + 2n^2\right).
\end{align*}
This is an increasing function for $j < n-3/2$, i.e., $j \leq n-2$ implying that $$t_2(\V)\geq \frac{n^2}{12}.$$

We now consider the case when the algorithm stops because $j \geq n/2$. Note that at this point, we will have pruned exactly $j$ points. Each pruned point $v_i$, $1 \leq i \leq j$, is incident to at least $\frac{n - i + 1}{2} > \frac{n - i}{2}$ ordinary lines determined by $\V_{i-1}$. The only way such an ordinary line is not ordinary in $\V$ is if it contains one of the previously pruned points. At most $i - 1 < i$ of the ordinary lines incident to $v_i$ contain other pruned points $v_k$, $k < i$. Therefore the total number of ordinary lines determined by $\V$ satisfies
\begin{align*}
t_2(\V) \geq \sum_{i=1}^j \frac{n - i}{2} - \sum_{i = 1}^j i = \frac{1}{2} \sum_{i=1}^j (n - 3i) \geq \frac{n^2 - 8n - 9}{16}.
\end{align*}
This gives us that for some absolute constant $c^\prime > 0$ and $n$ large enough, $$ t_2(\V) \geq c^\prime n^2. $$
\end{proof}

\section{A dependency matrix for a more refined bound}\label{sec:dependencymatrix}
In this section we give a refined construction for the dependency matrix of a point set $\V$. Recall that we defined the dependency matrix in Definition~\ref{def:dependency} to contain a row for each collinear triple from a triple system constructed on each special line. The goal was to not have too many triples containing the same pair (as can happen when there are many points on a single line). At the end of this section (Definition~\ref{def:dependency-second}) we will give a construction of a dependency matrix that will have an additional property (captured in Item 4 of Lemma~\ref{le:dmatrix-line}) which is used to obtain cancellation in the diagonal dominant argument, as outlined in the proof overview. 

We denote the argument of a non-zero complex number $z$ by $\arg{z}$, and use the convention that $\arg{z} \in (-\pi, \pi]$. 

\begin{definition}[angle between two complex numbers]
We define the {\em{angle between two non-zero complex numbers $a$ and $b$}} to be the the absolute value of the argument of $a\overline{b}$, denoted by $\left|\arg{a\overline{b}}\right|$.  Note that the angle between $a$ and $b$ equals the angle between $b$ and $a.$
\end{definition}

\begin{definition}[co-factor]
Let $v_1, v_2$ and $v_3$ be three distinct collinear points in $\CC^d$, and let $a_1, a_2$ and $a_3$ be the linear dependency coefficients among the three points.
Define the {\em co-factor} of $v_3$ with respect to $(v_1, v_2)$, denoted by $C_{1, 2}(3)$, to be $\frac{a_1\overline{a_2}}{|a_1||a_2|}$. Notice that this is well defined with respect to the points, and does not depend on the choice of coefficients.
\end{definition}

The next  lemma  will be used to show that ``cancellations" must arise in a line containing four points (as mentioned earlier in the proof overview). We will later use this lemma as a black box to quantify the cancellations in lines with more than four points by applying it to random four tuples inside the line.
\begin{lemma}
\label{le:coefficient-angle}
Let $v_1, v_2, v_3, v_4$ be four collinear points in $\CC^d$. Then at least one of the following holds:
\begin{enumerate}
\item The angle between $C_{1, 2}(3)$ and $C_{1, 2}(4)$ is at least $\pi/3$;
\item The angle between $C_{1, 3}(4)$ and $C_{1, 3}(2)$ is at least $\pi/3$;
\item The angle between $C_{1, 4}(2)$ and $C_{1, 4}(3)$ is at least $\pi/3$.
\end{enumerate}
\end{lemma}
\begin{proof}

For $i \in \{1, 2, 3, 4\}$, let $v_i^\prime = (v_i, 1)$, i.e., the vector obtained by appending 1 to $v_i$. Since $v_1, v_2, v_3, v_4$ are collinear, there exist $a_1, a_2, a_3 \in \CC$ such that
\begin{align}\label{eq:123}
 a_1 v^\prime_1 + a_2 v^\prime_2 +  a_3v^\prime_3 = 0,
\end{align}
and  $b_1, b_2, b_4 \in \CC$ such that
\begin{align} \label{eq:124}
 b_1 v^\prime_1 + b_2 v^\prime_2 +  b_4v^\prime_4 = 0.
\end{align}
We may assume, without loss of generality, that $a_3 = b_4 = 1$. Now equations \eqref{eq:123} and \eqref{eq:124} give us that
$C_{1,2}(3)=\frac{a_1\overline{a_2}}{|a_1||a_2|}$,
$C_{1,2}(4)=\frac{b_1\overline{b_2}}{|b_1||b_2|}$,
$C_{1,3}(2)=\frac{a_1}{|a_1|} $
and $C_{1,4}(2)=\frac{b_1}{|b_1|}.$

Combining equations \eqref{eq:123} and \eqref{eq:124}, we get the following linear equation:
\begin{align}\label{eq:134}
(b_2a_1 - b_1a_2) v^\prime_1 +b_2 v^\prime_3 -a_2  v^\prime_4 = 0. 
\end{align}
From (\ref{eq:134}), we get
$C_{1,3}(4)=\frac{(b_2a_1 - b_1a_2)\overline{b_2}}{|b_2a_1 - b_1a_2||b_2|}$
and $C_{1,4}(3)=-\frac{(b_2a_1 - b_1a_2)\overline{a_2}}{|b_2a_1 - b_1a_2||a_2|}.$

Then the angle between $C_{1, 2}(3)$ and $C_{1, 2}(4)$ is
\begin{align}
 \nonumber & \left|\arg{ \frac{a_1\overline{a_2}}{|a_1||a_2|} \frac{\overline{b_1}b_2}{|b_1||b_2|}}\right|\\
\label{eq:angle1} =\,\,&\left|\arg{ a_1\overline{a_2} \overline{b_1}b_2 }\right|.
\end{align}
The angle between $C_{1, 3}(4)$ and $C_{1, 3}(2)$ is
\begin{align}
\nonumber & \left|\arg{\frac{(b_2a_1 - b_1a_2)\overline{b_2}}{|b_2a_1 - b_1a_2||b_2|}\frac{\overline{a_1}}{|a_1|}}\right|  \\ 
\label{eq:angle2}  = \,\,& \left|\arg{\overline{ a_1 }\overline{b_2}(b_2a_1 - b_1a_2) }\right|.
\end{align}
The angle between $C_{1, 4}(2)$ and $C_{1, 4}(3)$ is
\begin{align}
\nonumber & \left|\arg{-\frac{b_1}{|b_1|}\frac{\overline{(b_2a_1 - b_1a_2)}a_2}{|b_2a_1 - b_1a_2||a_2|}}\right| \\
\label{eq:angle3}  =\,\, & \left|\arg{ -b_1a_2\overline{(b_2a_1 - b_1a_2)} }\right|.
\end{align}

Note that the product of expressions inside the arg functions in  (\ref{eq:angle1}), (\ref{eq:angle2}) and (\ref{eq:angle3}) is a negative real number, and so the sum of (\ref{eq:angle1}), (\ref{eq:angle2}) and (\ref{eq:angle3}) must be $\pi$. It follows that one of the angles must be at least $\pi/3$.
\end{proof}

Our final dependency matrix will be composed of blocks, each given by the following lemma. Roughly speaking, we construct a block  of rows $A(l)$ for each special line $l$. The rows in $A(l)$ will be chosen carefully and will correspond to triples that will eventually give non trivial cancellations.

\begin{lemma}
\label{le:dmatrix-line}
Let $l$ be a line in $\CC^d$ and $\V_l = \{v_1, \dots v_r\}$ be points on $l$ with $r \geq 3$. Let $V_l$ be the $r \times (d+1)$ matrix whose $i^{th}$ row is the vector $(v_i, 1)$. Then there exists an $(r^2 - r) \times r$ matrix $A = A(l)$, which we refer to as the {\em dependency matrix} of $l$, such that the following hold:
\begin{enumerate}
\item $A V_l = 0$;
\item Every row of $A$ has support of size three;
\item The support of every two columns of $A$ intersects in exactly six locations;
\item If $r\geq 4$ then for at least $1/3$ of choices of $k \in [r^2-r]$, there exists $k^\prime \in [r^2-r]$ such that following holds: For $k \in [r^2-r]$, let $R_k$ denote the $k^{th}$ row of $A$; Suppose $\supp(R_k) = \{i, j, s\}$. Then $\supp(R_{k^\prime}) = \{i, j, t\}$ (for some $t \neq s$) and the angle between the co-factors $C_{i,j}(s)$ and $C_{i,j}(t)$ is at least $\pi/3$.
\end{enumerate}
\end{lemma}
\begin{proof}
Recall that Lemma~\ref{le:triplesystem} gives us a family of triples $T_r$ on the set $[r]^3$.  For every bijective map $\sigma : \V_l \rightarrow [r]$, construct a matrix $A_\sigma$ in the following manner: Let $T_l$ be the triple system on $\V_l^3$ induced by $\sigma$ and $T_r$. For each triple $(v_i, v_j, v_k) \in T_l$, add a  row with three non-zero entries in positions $i, j, k$ corresponding to the linear dependency coefficients between $v_i, v_j$ and $v_k$. 

Note that for every $\sigma$, $A_\sigma$ has $r^2 - r$ rows and $r$ columns. Since the rows correspond to linear dependency coefficients, clearly we have $A_\sigma V_l = 0$ satisfying Property 1. Properties 2 and 3 follow from properties of the triple system from Lemma~\ref{le:triplesystem}.

We will use a probabilistic argument to show that there exists a matrix $A$ that has Property $4$. Let $\Sigma$ be the collection of all bijective maps from $[r]$ to the points $\V_l$, and let $\sigma \in \Sigma$ be a uniformly random element. Consider $A_\sigma$.
Since every pair of points occurs in at least two distinct triples, for every row $R_k$ of $A_\sigma$, there exists a row $R_{k^\prime}$ such that the supports of $R_{k}$ and $R_{k^\prime}$ intersect in two entries. Suppose that $R_k$ and $R_{k^\prime}$ have supports contained in $\{ i , j, s, t\}$.  
Suppose that $\sigma$ maps $\{ v_i, v_j, v_s, v_t\}$ to $\{1, 2, 3, 4\}$ and that $(1,2,3)$ and $(1,2,4)$ are triples in $T_r.$  Without loss of generality, assume $v_i$ maps to 1. Then by Lemma~\ref{le:coefficient-angle}, the angle between at least one of the pairs $\{ C_{i, j}(s), C_{i,j}(t) \}$, $\{ C_{i, s}(j), C_{i, s}(t) \}$, $\{ C_{i, t}(j), C_{i, t}(s) \}$ must be at least $\pi/3$.  That is, given that $v_i$ maps to $1$, we have that the probability that $R_k$ satisfies Property 4 is at least $1/3$. Then it is easy to see that
\[ \Pr(R_k \mbox{ satisfies Property }4) \geq 1/3. \]

Define the random variable $X$ to be the number of rows satisfying Property 4, and note that we have \[\mathbb{E}[X] \geq (r^2 - r)\frac{1}{3}. \]
It follows that there exists a matrix $A$ in which at least $1/3$ of the rows satisfy Property 4.
\end{proof}

To argue about the off-diagonal entries of $M$, we will use the following notion of balanced rows. The main idea here is that, if there are many rows that are not balanced then we win in one of the Cauchy-Schwartz applications and, if many rows are balanced then we win from cancellations that show up via the different angles.

\begin{definition}[$\eta$-balanced row]
Given an $m\times n$ matrix $A$, we say a row $R_k$ is $\eta$-balanced for some constant $\eta$ if  $\left||A_{ki}|^2 - |A_{kj}|^2\right| \leq \eta$, for every $i, j \in \supp(R_k)$. Otherwise, we say that $R_k$ is $\eta$-unbalanced. When $\eta$ is clear from the context, we say that the row is balanced/unbalanced.
\end{definition}

\begin{lemma}
\label{le:squares-bound}
There exists an absolute constant $c_0 > 0$ such that the following holds. Let $l$ be a line in $\CC^d$ and $\V_l = \{v_1, \dots v_r\}$ be points on $l$ with $r \geq 4$. Let $A = A(l)$ be the dependency matrix for $l$, defined in Lemma~\ref{le:dmatrix-line}, and $A^\prime$ a scaling of $A$ such that the $\ell_2$ norm of every row is $\alpha$. Let $M = A^{\prime*}A^\prime$. 
\[ \sum_{i \neq j} |M_{ij}|^2 \leq 4(r^2 - r)\alpha^4 - c_0 (r^2 - r)  \alpha^2. \]
\end{lemma}
\begin{proof}
Recall that $A$ is an $(r^2-r) \times r$ matrix, that the support of every row has size exactly three, and that the supports of any two distinct columns of $A$ intersects in six locations. Clearly, any scaling $A^\prime$ of $A$ will also satisfy these properties. Applying Lemma~\ref{le:offdiagonalsum} to $A^\prime$ gives
\begin{equation}
\label{eq:sumsq}
\sum_{i \neq j}  |M_{ij}|^2  =  4(r^2 - r)\alpha^4 - \left(D(A^\prime) + 2E(A^\prime)\right).
\end{equation}

We are able to give a lower bound on $D(A^\prime) + 2E(A^\prime)$ using Property 4 of Lemma~\ref{le:dmatrix-line}. From here on, we focus on the rows mentioned in Property 4. Recall that there are at least $(r^2 - r)/3$ such rows. For some $\eta$ to be determined later, suppose that $\beta$ fraction of these rows is $\eta$-unbalanced. We will show each such row contributes to either $D(A^\prime)$ or $E(A^\prime)$.

If a row $R_k$ is $\eta$-imbalanced, we have \[ \sum_{i < j}\left( |A^\prime_{ki}|^2 - |A^\prime_{kj}|^2\right)^2 > \eta^2. \]
Alternatively suppose that $R_k$ is $\eta$-balanced. Recall that $\sum_{i = 1}^{n} |A^\prime_{ki}|^2 = \alpha,$ and note that we must have that $|A^\prime_{ki}|^2 \in [\frac{\alpha}{3} - \frac{2\eta}{3}, \frac{\alpha}{3} + \frac{2\eta}{3}]$ for all $i \in \supp(R_k)$. Suppose that both $R_k$ and $R_{k^\prime}$ have non-zero entries in columns $i$ and $j$, but $R_k$ has a third nonzero entry in column $s$ and $R_{k^\prime}$ has a third nonzero entry in column $t$, where $s\neq t$. Suppose further that the angle $\theta$ between the co-factors $C_{i,j}(s)$ and $C_{i,j}(t)$ is at least $\pi/3$, i.e., $\cos \theta \leq 1/2$. Then
\begin{align*}
& \left| A^\prime_{ki} \overline{A^\prime_{kj}} - A^\prime_{k^\prime i} \overline{A^\prime_{k^\prime j}} \right|^2 \\
= \,\,& |A^\prime_{ki} \overline{A^\prime_{kj}} |^2 + |A^\prime_{k^\prime i} \overline{A^\prime_{k^\prime j}} |^2 - 2 |A^\prime_{ki} \overline{A^\prime_{kj}} ||A^\prime_{k^\prime i} \overline{A^\prime_{k^\prime j}} | \cos \theta \\
\geq \,\,& |A^\prime_{ki} \overline{A^\prime_{kj}} |^2 + |A^\prime_{k^\prime i} \overline{A^\prime_{k^\prime j}} |^2 - |A^\prime_{ki} \overline{A^\prime_{kj}} ||A^\prime_{k^\prime i}\overline{A^\prime_{k^\prime j}} |.
\end{align*}
Any positive real numbers $a,b$ satisfy
\begin{equation*}
a^2+b^2 - ab = 
\left(\frac{a}{2} - b\right)^2 + \frac{3}{4}a^2   \geq  \frac{3}{4}a^2.
\end{equation*}
 Substituting $a=|A^\prime_{ki} \overline{A^\prime_{kj}} |$ and $b= |A^\prime_{k^\prime i} \overline{A^\prime_{k^\prime j}}| $ gives
\begin{align*}
& |A^\prime_{ki} \overline{A^\prime_{kj}} |^2 + |A^\prime_{k^\prime i} \overline{A^\prime_{k^\prime j}} |^2 - |A^\prime_{ki} \overline{A^\prime_{kj}} ||A^\prime_{k^\prime i} \overline{A^\prime_{k^\prime j}} | \\
\geq \,\,& \frac{3}{4}|A^\prime_{ki} \overline{A^\prime_{kj}}|^2 \\
\geq \,\,& \frac{3}{4} \left(\frac{\alpha}{3} - \frac{2\eta}{3}\right)^2 \\
= \,\,& \frac{1}{12} \left( \alpha - 2\eta \right)^2.
\end{align*}
Summing over the $\eta$-unbalanced rows, we have 
\begin{align*}
E(A^\prime) & \geq \beta  \frac{(r^2 - r)}{3} \eta^2. 
\end{align*}
Summing over all the  $\eta$-balanced rows gives
\begin{align*}
D(A^\prime) & = \sum_{i \neq j} \sum_{k < k^\prime} \left| A^\prime_{ki}\overline{A^\prime_{kj}} - A^\prime_{k^\prime i}\overline{A^\prime_{k^\prime j}} \right|^2\\
& = \frac{1}{2} \sum_{k \neq k^\prime} \sum_{i \neq j}  \left| A^\prime_{ki}\overline{A^\prime_{kj}} - A^\prime_{k^\prime i}\overline{A^\prime_{k^\prime j}} \right|^2\\
& \geq \frac{1}{2} \cdot (1 - \beta)\frac{(r^2 - r)}{3} \cdot \frac{1}{12}\left(\alpha - 2\eta \right)^2. \\
& = (1 - \beta)\frac{(r^2 - r)}{72} \left(\alpha - 2\eta \right)^2.
\end{align*}
Combining the lower bounds for $D(A)$ and $E(A)$, and setting $\eta = \alpha/10$ gives
\begin{align*}
D(A^\prime) + 2E(A^\prime) & \geq  (1 - \beta)\frac{(r^2 - r)}{72} \left(\alpha - 2\eta \right)^2 +  2 \beta  \frac{(r^2 - r)}{3} \eta^2\\
& =  (r^2 - r) \left ( (1 - \beta)\frac{1}{72} \left(\frac{4}{5}\alpha \right)^2 +  \beta \frac{2}{3} \left( \frac{1}{10}\alpha \right)^2 \right)\\
& \geq c_0 (r^2 - r) \alpha^2
\end{align*}
for some absolute constant $c_0$.
Combining the above with Equation (\ref{eq:sumsq}), we have
\[ \sum_{i \neq j} |M_{ij}|^2 \leq 4(r^2 - r)\alpha^4 - c_0 (r^2 - r)  \alpha^2. \]
\end{proof}

We are now ready to define the full dependency matrix that we will use in the proof of Theorem~\ref{th:main}. 

\begin{definition}[Dependency Matrix, second construction]\label{def:dependency-second}
Let $\V = \{v_1, \dots v_n\}$ be a set of $n$ points in $\CC^d$ and let $V$ be the $n \times (d+1)$ matrix whose $i^{th}$ row is the vector $(v_i, 1)$.  For each matrix $A(l)$, where $l \in \L_{\geq 3}(\V)$, add $n - r$ column vectors of all zeroes, with length $r^2-r$, in the column locations corresponding to points not in $l$, giving an $(r^2 - r) \times n$ matrix. Let $A$ be the matrix obtained by taking the union of rows of these matrices for every $l \in \L_{\geq 3}(\V)$. We refer to $A$ as the {\em dependency matrix} of $\V$. 

\end{definition}

Note that this construction is a special case of the one given in Definition~\ref{def:dependency} and so satisfies all the properties mentioned there. In particular, we have $AV=0$ and the number of rows in $A$ equals $ n^2 - n - 2t_2(\V)$.

%%%%%%%

\section{Proof of Theorem~\ref{th:main}}\label{sec:main}

Before we prove the theorem, we give some key lemmas. As before, we consider two cases: When the dependency matrix $A$ satisfies Property-$S$ and when it does not. In the latter case, we rely on Lemma~\ref{le:not-propertys}. The following lemma deals with the former case.

\begin{lemma}\label{th:gap}
There exists an absolute constant $c_1 > 0$ such that the following holds. Let $\V= \{v_1, v_2, \dots, v_n\}$ be a set of points in $\CC^d$ not contained in a plane. Let $A$ be the $m \times n$ dependency matrix for $\V$, and suppose that $A$ satisfies Property-$S$. Then
\[t_2(\V) \geq \frac{3}{2}n + c_1 \sum_{r \geq 4} (r^2 - r) t_r(\V). \]
\end{lemma}

\begin{proof}
Since $A$ satisfies Property-$S$, by Corollary~\ref{co:complexscaling} for every $\epsilon > 0$, there exists a scaling $A^\prime$ of $A$ such that for every $i \in [m]$
\begin{equation*}
\sum_{j \in [n]} \left|A^\prime_{ij}\right|^2 = 1 + \epsilon,
\end{equation*}
and for every $j \in [n]$
\begin{equation}
\label{eq:columnsum}
\sum_{i \in [m]} \left|A^\prime_{ij}\right|^2 \geq \frac{m}{n} - \epsilon.
\end{equation}

Let $C_i$ be denote the $i^{th}$ column of $A^\prime$, and let $M = A^{\prime*}A^\prime$. From (\ref{eq:columnsum}), we have $|M_{ii}|  = \langle C_i, C_i \rangle \geq \left(\frac{m}{n} - \epsilon \right)$.

To bound the sum of squares of the off-diagonal entries, we go back to the construction of the dependency matrix. Recall that the matrix $A$ was obtained by taking the union of rows of matrices $A(l)$, for each $l \in \L_{\geq 3}$. Then we have that $A^\prime$ is the union of scalings of the rows of the matrices $A(l)$, for each $l \in \L_{\geq 3}$. Note that $|M_{ij}| = \langle C_i, C_j\rangle$ and that the intersection of the supports of any two distinct columns in contained within a scaling of $A(l)$, for some $l \in \L_{\geq 3}$. Therefore, to get a bound on $\sum_{i \neq j}|M_{ij}|^2$, it suffices to consider these component matrices.
Combining the bounds obtained from Lemma~\ref{le:squares-bound}, with $\alpha = 1 + \epsilon$, we have 
\begin{align*}
\sum_{i \neq j} |M_{ij}|^2 & \leq \sum_{l \in \L_{3}} 4(r^2 - r)\alpha^4 + \sum_{l \in \L_{\geq 4}} \left( 4(r^2 - r)\alpha^4 - c_0 (r^2 - r)  \alpha^2\right) \\
& =  \sum_{l \in \L_{\geq 3}} 4(r^2 - r)\alpha^4 - \sum_{l \in \L_{\geq 4}} c_0 (r^2 - r)  \alpha^2 \\
& = 4m(1 + \epsilon)^4 - (1 + \epsilon)^2 c_0 \sum_{r \geq 4} (r^2 - r) t_r.
\end{align*}
Let $F = c_0 \sum_{r \geq 4}^{n} (r^2 - r) t_r$.
Lemma~\ref{le:rankbound} implies
\begin{align*}
\rank(M) & \geq \frac{n^2L^2}{nL^2 + \sum_{i \neq j} |M_{ij}|^2}\\
& \geq \frac{n^2\left(\frac{m}{n} - \epsilon\right)^2}{n\left(\frac{m}{n} - \epsilon\right)^2 +  4m(1 + \epsilon)^4 - (1 + \epsilon)^2 F}.
\end{align*}
Taking $\epsilon$ to 0, we get
\begin{align*}
\rank(M) & \geq \frac{n^2\left(\frac{m}{n}\right)^2}{n\left(\frac{m}{n}\right)^2 +  4m -   F} \\
& = n - \frac{4n^2m - n^2 F}{m^2 +  4mn -  n F}.
\end{align*}
Since \[ \adim(\V) = \rank(V) - 1 \leq \frac{4n^2m - n^2 F}{m^2 +  4mn -  n F} - 1, \] 
if
\begin{align*}
\frac{4n^2m - n^2 F}{m^2 +  4mn -  n F} & < 4,
\end{align*}
then $\V$ must be contained in a plane, contradicting the assumption of the theorem. Substituting $m = n^2 - n - 2t_2(\V)$ and simplifying gives
\begin{align*}
4 t_2^2 - (2 n^2  + 4 n )t_2 + 3 n^3 - 3 n^2 + \frac{n^2F}{4}  - nF   & > 0.
\end{align*}
This holds when
\begin{align*}
t_2(\V)& < \frac{3n}{2} + \frac{F}{8} = \frac{3n}{2} +  \frac{c_0}{8} \sum_{r = 4}^{n} (r^2 - r) t_r(\V),
\end{align*}
which completes the proof.
\end{proof}

We now have the following easy corollary.
\begin{corollary}\label{th:eitherorgap}
There exists a positive integer $n_0$ such that the following holds. Let $c_1$ be the constant from Lemma~\ref{th:gap} and let $\V$ be a set of $n \geq n_0$ points in $\CC^d$ not contained in a plane. Then at least one of the following holds:
\begin{enumerate}
\item There exists a point $v \in \V$ contained in at least $\frac{n}{2}$ ordinary lines;
\item $ t_2(\V) \geq \frac{3}{2}n + c_1\sum_{r\geq 4}(r^2-r)t_r(\V)$.
\end{enumerate}
\end{corollary}

\begin{proof}
If $A$ satisfies Property-$S$, then we are done by Lemma~\ref{th:gap}.  Otherwise, let $b^*$ be an integer such that
\begin{equation}
\label{eq:b*}
\frac{n}{2}(b^*-1)< \frac{3n}{2}+ c_1 \sum_{r\geq 4}(r^2-r)t_r(\V) \leq \frac{n}{2}b^*.
\end{equation}

Clearly we have $b^* > 1$. Recall that $\sum_{r\geq 4}(r^2-r)t_r(\V) < n^2,$ implying that for $c_1$ small enough and $n$ large enough, 
\begin{equation}
\label{eq:b*upperbound}
{ b^* < 4 +\frac{2c_1}{n}\sum_{r\geq 4}(r^2-r)t_r(\V) < \frac{1}{6}n}.
\end{equation}
Now by Lemma~\ref{le:not-propertys} and (\ref{eq:b*}), either the number of ordinary lines
\[t_2(\V) \geq \frac{n}{2}b^* \geq \frac{3n}{2}+ c_1 \sum_{r\geq 4}(r^2-r)t_r(\V), \]
or, using (\ref{eq:b*upperbound}), there exists a point $v \in \V$, such that the number of ordinary lines containing $v$ is at least
\[ \frac{2}{3}(n+1) - b^* > \frac{1}{2}n. \]

\end{proof}

The following lemma will be crucially used in the proof of Theorem~\ref{th:main}.
\begin{lemma}\label{le:sumVprime}
Let $\V$ be a set of $n$ points in $\CC^d$, and $\V^\prime = \V \setminus \{v\}$ for some $v \in \V$. Then 
\[ \sum_{r \geq 4}(r^2 - r)t_r(\V^\prime) \geq \sum_{r \geq 4}(r^2 - r)t_r(\V) - 4 (n - 1) .\]
\end{lemma}
\begin{proof}
Note that when we remove $v$ from the set $\V$, we only affect lines that contain $v$. In particular, ordinary lines containing $v$ are removed and the number of points on every special line containing $v$ goes down by 1. Every other line remains unchanged and so it suffices to consider only lines that contain the point $v$. 

We consider the difference  \[ K = \sum_{r \geq 4}(r^2 - r)t_r(\V) - \sum_{r \geq 4}(r^2 - r)t_r(\V^\prime). \]

We will consider the contribution of a line $l$ determined by $\V$ to the difference $K$.

Each line $l \in \L_{\geq 5}(\V)$, i.e., a line that has $r \geq 5$ points, that contains $v$ contributes $r^2 - r$ to the summation $\sum_{r \geq 4}(r^2 - r)t_r(\V)$. In $\V^\prime$, $l$ has $r - 1$ points, and contributes $(r-1)^2 - (r-1)$ to the summation $\sum_{r \geq 4}(r^2 - r)t_r(\V^\prime)$.  Therefore, $l$ contributes $2(r-1)$ to the difference $K$. We may charge this contribution to the points on $l$ that are not $v$. There are $r-1$ other points on $l$, so each point contributes $2$ to $K$.

Each line $ l \in \L_{4}(\V)$ that contains $v$ contributes $r^2 - r = 12$ to the summation $\sum_{r \geq 4}(r^2 - r)t_r(\V)$. These lines contain three points in $\V^\prime$, and so do not contribute anything in the $\sum_{r \geq 4}(r^2 - r)t_r(\V^\prime)$ term. Once again, we charge this contribution to the points lying on $l$ that are not $v$. Each such line has three points on it other than $v$, so each point contributes $12/3 = 4$ to $K$.

There is a unique line containing $v$ and any other point, and each point either contributes $0$, $2$ or $4$ to $K$. This gives \[ \sum_{r \geq 4}(r^2 - r)t_r(\V) - \sum_{r \geq 4}(r^2 - r)t_r(\V^\prime) \leq 4(n-1). \]
Rearranging completes the proof.
\end{proof}

We are now ready to prove the main theorem. For convenience, we restate the theorem.
\begin{reptheorem}{th:main}
There exists an absolute constant $c > 0$ and a positive integer $n_0$ such that the following holds. Let  $\V$ be a set of $n \geq n_0$ points in $\CC^3$ with at most $\frac{1}{2}n$ points contained in any plane. Then
\[ t_2(\V) \geq \frac{3}{2}n + c \sum_{r \geq 4} r^2 t_r(\V). \]
\end{reptheorem}

\begin{proof}

The remainder of the proof is similar to the proof of Theorem~\ref{th:higherdim}, i.e., we use Corollary~\ref{th:eitherorgap} to find a point incident to a large number of ordinary lines,  ``prune'' this point, and then repeat this on the smaller set of points. We stop when either we can not find such a point, in which case Corollary~\ref{th:eitherorgap} guarantees a large number of ordinary lines, or when we have accumulated enough ordinary lines.

As before, consider the following algorithm:
Let $\V_0 := \V$ and $j = 0$.
\begin{enumerate}
\item If $\V_j$ satisfies case (2) of Lemma~\ref{th:eitherorgap}, then stop.
\item Otherwise, there must exist a point $v_{j+1}$ incident to at least $\frac{n - j}{2}$ ordinary lines determined by $\V_j$. Let $\V_{j+1} = \V_j \setminus \{ v_{j+1} \}$.
\item Set $j = j + 1$. If $j \geq n/2$, then stop. Otherwise, go to Step 1.
\end{enumerate}

Note that since no plane contains more than $n/2$ points, at no point will the algorithm stop because the configuration becomes planar. That is, we can use Corollary~\ref{th:eitherorgap} at every step of the algorithm. We now analyze the two stopping conditions for the algorithm, and show that we can always find enough ordinary lines by the time the algorithm stops.

Suppose that we stop because $\V_j$ satisfies case (2) of Corollary~\ref{th:eitherorgap} for some $1 \leq j < n/2$.  From case (2) of Lemma~\ref{th:eitherorgap} and Lemma~\ref{le:sumVprime}, we have  
\begin{align}
\nonumber t_2(\V_j) & \geq \frac{3(n - j)}{2}+c_1\sum_{r\geq 4} (r^2-r)t_r(\V_j)\\
\label{eq:ordinaryvj}& \geq\frac{3(n - j)}{2} + c_1\left(\sum_{r\geq 4} (r^2-r)t_r(\V) - 4 \sum_{i = 1}^j (n-i)\right).
\end{align}
On the other hand, each pruned point $v_i$, $1 \leq i \leq j$, is incident to at least $\frac{n - i + 1}{2} > \frac{n - i}{2}$ ordinary lines determined by $\V_{i-1}$, and hence, at most $(n - i - \frac{n - i + 1}{2})/2 < \frac{n - i}{4}$ special lines.  Note that an ordinary line in $\V_i$ might not be ordinary in $\V_{i-1}$ if contains $v_{i}$. Thus, in order to lower bound the total number of ordinary lines in $\V$, we sum over the number of ordinary lines contributed by each of the pruned points $v_{i}$,
$1 \leq i \leq j$, and subtract from the count the number of potential lines that could contain $v_{i}$. Then the number of ordinary lines contributed by the pruned points is at least
\begin{align}
\label{eq:ordinarypruning}  \sum_{i=1}^j \left(\frac{n - i}{2} - \frac{n - i}{4} \right) = \frac{1}{4}\sum_{i=1}^j \left(n - i\right).
\end{align}
Combining (\ref{eq:ordinaryvj}) and (\ref{eq:ordinarypruning}), we have 
\begin{align*}
t_2(\V) & \geq \frac{3}{2}(n-j) +c_1\left(\sum_{r\geq 4}(r^2-r)t_r(\V) - 4 \sum_{i = 1}^j (n-i) \right) + \frac{1}{4}\sum_{i=1}^j \left(n - i\right)\\
& = \frac{3}{2}n + c_1 \sum_{r\geq 4}(r^2-r)t_r(\V) + \left(\frac{1}{4} - 4c_1\right) \sum_{i = 1}^j (n-i) - \frac{3}{2}j.
\end{align*}
For $c_1$ small enough and $n$ large, the term $\left(\frac{1}{4} - 4c_1\right) \sum_{i = 1}^j (n-i) - \frac{3}{2}j$ is positive.  Therefore, there exists some absolute constant $c > 0$ such that 
\[ t_2(\V) \geq \frac{3}{2}n + c\sum_{r\geq 4}r^2t_r(\V). \]

We now consider the case when the algorithm stops because $j \geq n/2$. Note that at this point, we will have pruned exactly $j$ points. Each pruned point $v_i$, $1 \leq i \leq j$, is incident to $\frac{n - i + 1}{2} > \frac{n - i}{2}$ ordinary lines determined by $\V_{i-1}$. However, as many as $i - 1 < i$ of these lines could contain other pruned points $v_k$, $k < i$, i.e., lines that could be special in $\V$. Therefore the total number of ordinary lines determined by $\V$ is at least
\begin{align*}
t_2(\V) \geq \sum_{i=1}^j \frac{n - i}{2} - \sum_{i = 1}^j i = \frac{1}{2} \sum_{i = 1}^j (n - 3 i) \geq \frac{n^2 - 8n - 9}{16}.
\end{align*}
Note that $n^2 > \sum_{r\geq 4}(r^2-r)t_r(\V)$, which gives 
\[ t_2(\V)\geq\frac{3}{2}n+c\sum_{r\geq 4}r^2t_r(\V) \]
for some absolute constant $c > 0$ and $n$ large enough.

\end{proof}

\bibliographystyle{alpha}
\bibliography{gapproof}

\end{document}